\documentclass[12pt,twoside,a4paper]{article}

\usepackage[utf8]{inputenc}    

\usepackage{graphicx}
\usepackage{amsmath,amsfonts,amsthm}
\usepackage{array}
\usepackage{dcolumn}
\usepackage{psfrag}
\usepackage{latexsym}

\usepackage{array}
\usepackage{dcolumn}

\usepackage{epsfig}
\usepackage{enumerate}
\usepackage{psfrag}
\usepackage{graphics}
\usepackage{amsmath,amsfonts,amsthm}

\usepackage{url}

\usepackage{epstopdf}
\usepackage[usenames,dvipsnames]{xcolor}
\usepackage[T1]{fontenc}
\makeatletter
\newcommand\footnoteref[1]{\protected@xdef\@thefnmark{\ref{#1}}\@footnotemark}
\makeatother

\usepackage[square,numbers]{natbib}
\usepackage{authblk} 

\title{A note on the asymptotic expansion of the Lerch's transcendent}

\author[1]{Xing Shi Cai\thanks{This work is supported by the Knut and Alice Wallenberg Foundation and the \emph{Ministerio de Econom\'{\i}a y Competitividad} of the spanish government (MTM2017-83490-P). Email: \texttt{xingshi.cai@math.uu.se}}}

\author[2]{José L. López\thanks{Email: \texttt{jl.lopez@unavarra.es}}}

\affil[1]{Department of Mathematics, Uppsala University, Uppsala, Sweden}
\affil[2]{Departamento de Estad\'{\i}stica, Matemáticas e Informática and INAMAT, Universidad Pública de Navarra, Pamplona, Spain}

\usepackage{mathtools}
\mathtoolsset{showonlyrefs}

\usepackage{stmaryrd}

\newcommand{\Z}{{\mathbb Z}}
\newcommand{\N}{{\mathbb N}}
\newcommand{\C}{{\mathbb C}}

\newcommand\bigO[1]{{\mathcal{O}\!\left( #1 \right)}}

\newcommand{\polylog}{\mathrm{Li}}

\newcommand{\pochhammer}[2]{\left(#1\right)_{#2}}

\newcommand{\eqd}{\coloneqq}

\newtheorem{lemma}{Lemma}

\newtheorem{theorem}{Theorem}

\newtheorem{corollary}{Corollary}

\theoremstyle{definition}

\newtheorem{myRemark}{Remark}
\newtheorem*{myRemark*}{Remark}

\usepackage{color}
\definecolor{webgreen}{rgb}{0,.5,0} 
\definecolor{Maroon}{cmyk}{0, 0.87, 0.68, 0.32}
\usepackage{hyperref} 
\hypersetup{colorlinks, breaklinks, urlcolor=Maroon, linkcolor=Maroon, citecolor=webgreen} 

\usepackage{tikz}
\usetikzlibrary{angles,fit,arrows}
\usepackage[skip=12pt]{caption}

\begin{document}

\maketitle

\begin{abstract}
In \cite{MR2086542}, the authors derived an asymptotic expansion of the Lerch's transcendent
\(\Phi(z,s,a)\) for large $\vert a\vert$, valid for $\Re a>0$, $\Re s>0$ and
$z\in\mathbb{C}\setminus[1,\infty)$. In this paper we study the special case $z\ge 1$ not covered in
\cite{MR2086542}, deriving a complete asymptotic expansion of the Lerch's transcendent
\(\Phi(z,s,a)\) for \(z > 1\) and $\Re s>0$ as $\Re a$ goes to infinity.  We also show that when
\(a\) is a positive integer, this expansion is convergent for $\Re z\ge 1$.  As a corollary, we get a full asymptotic
expansion for the sum \(\sum_{n=1}^{m} z^{n}/n^{s}\) for fixed \(z >1 \) as \(m \to \infty\). Some
numerical results show the accuracy of the approximation.
\end{abstract}

\begin{section}{Introduction}
The Lerch's transcendent (Hurwitz-Lerch zeta function) \cite[\S25.14(i)]{DLMFlerch} is defined by means of the power series
\begin{equation}
    {\Phi\left(z,s,a\right)=\sum_{n=0}^{\infty}\frac{z^{n}}{(a+n)^{s}}}, \hskip 2cm a \ne 0, -1, -2,\dots,
    \label{eq:Lerch}
\end{equation}
on the domain $|z| < 1$ for any $s\in\mathbb{C}$ or $|z|\le 1$ for $\Re s>1$.
For other values of the variables \(z, s, a\), the function \(\Phi(z,s,a)\) is defined by analytic continuation.  In particular \cite{MR2086542},
\begin{equation}\label{inte}
\Phi\left(z,s,a\right)=\frac{1}{\Gamma(s)}\int_0^\infty\frac{x^{s-1}e^{-ax}}{1-ze^{-x}}\,\mathrm{d}x, \hskip 5mm
\Re a>0, \hskip 2mm z\in\mathbb{C}\setminus[1,\infty) \hskip 2mm \text{ and}\hskip 2mm\Re s>0.
\end{equation}

This function was investigated by Erd\'{e}lyi \cite[\S 1.11, eq. 1]{bateman}. Although using a
different notation $z=e^{2\pi i x}$, it was previously introduced by Lerch \cite{lerch} and Lipschitz
\cite{lipschitz} in connection with Dirichlet's famous theorem on primes in arithmetic progression.
If $x\in \mathbb{Z}$, the Hurwitz-Lerch zeta function reduces to the meromorphic Hurwitz zeta
function $\zeta(s,a)$ \cite[\S 2.3, eq. 2]{sriva}, with one single pole at $s=1$. Moreover,
$\zeta(s,1)$ is nothing but the Riemann zeta function $\zeta(s)$.

Properties of the Lerch's transcendent have been studied by many authors. Among other results, we remark the
following ones. Apostol obtains functional relations for $\Phi\left(e^{2\pi i x},s,a\right)$ and gives an algorithm to compute $\Phi\left(e^{2\pi i x},-n,a\right)$ for $n\in\mathbb{N}$ in terms of a certain kind of generalized Bernoulli polynomials \cite{apostol}. The function $\Phi\left(e^{2\pi i x},s,a\right)$ is used in \cite{katsuradaii} to generalize a certain
asymptotic formula considered by Ramanujan. Asymptotic equalities for some weighted mean squares of $\Phi\left(e^{2\pi i x},s,a\right)$ are given in \cite{kluss}. Integral representations, as well as functional relations
and expansions for $\Phi(z,s,a)$ may be found in \cite[\S 2.5]{sriva}. See \citet{MR0058756} for further properties. Here we want to remark the following two important properties of the Lerch's transcendent valid for $x,z,s \in \C$, $m \in \N$ \cite[\S 25.14.3 and \S 25.14.4]{DLMFlerch}:
\begin{equation}
    \Phi\left(z,s,1\right)=\frac{1}{z}\mathrm{Li}_{s}\left(z\right):=\sum_{n=1}^{\infty}\frac{z^{n-1}}{n^{s}}, \hskip 3cm \vert z\vert<1,
    \label{eq:polylog:def}
\end{equation}
where \(\polylog_{s}(z)\) is the polylogarithm function \cite[\S 25.12]{DLMFlerch}, and
\begin{equation}
    \Phi\left(z,s,x\right)=z^{m}\Phi\left(z,s,x+m\right)+\sum_{n=0}^{m-1}\frac{z^{n}}{(x+n)^{s}},
    \qquad 
    -x\notin\N\cup\lbrace 0\rbrace
    .
    \label{eq:lerch:id}
\end{equation}

In particular, when $x=1$, the second property may be written in the form
\begin{equation}
    \eta(z,s,m)
    \eqd
    \sum_{n=1}^{m} \frac{z^{n}}{n^{s}} 
    = 
    \polylog_{s}(z)-
    {z^{m+1}} \Phi\left( z, s, m+1 \right).
    \label{eq:lerch:sum}
\end{equation}
The finite sum $\eta(z,s,m)$ for $z>1$ is of interest in the study of \emph{random records in
full binary trees} by \citet{janson04}. In a full binary tree, each node has two child nodes
and each level of the tree is full. Thus \(T_{m}\), a full binary tree of height \(m\), has
\(n=2^{m+1}-1\) nodes. In the random records model, each node in \(T_{m}\) is given a label chosen
uniformly at random from the set \(\{1,\dots,n\}\) without replacement. A node \(u\) is called a \emph{record} when
its label is the smallest among all the nodes on the path from \(u\) to the root node. Let \(h(u)\)
be the distance from \(u\) to the root.  Let \(X(T_{m})\) be the (random) number of records in
\(T_{m}\). Then it is easy to see that the expectation of \(X(T_{m})\) is simply
\begin{equation}
    \sum_{u \in T_{n}}
    \frac{1}{h(u)+1}
    =
    \sum_{i=0}^{m+1}
    \frac{2^{i}}{i+1}
    =
    \frac{\eta\left(2,1,m+1 \right)}{2}
    =
    \frac{2^{m+1}}{m}
    +
    \bigO{\frac{2^{m}}{m^2}}
    =
    \frac{n}{m}
    +
    \bigO{
        \frac{n}{m^{2}}
    }
    ,
    \label{eq:X:T}
\end{equation}
where the last step follows from elementary asymptotic computations \cite[Remark~1.3]{janson04}.
A generalization of random records, called \emph{random \(k\)-cuts}, requires a similar computation 
which boils down to finding an asymptotic expansion of \(\eta(2, b/k, m)\) for some \(k \in \N\)
and \(1 \le b \le k\) as \(m \to \infty\), see \cite[\S 5.3.1]{Cai010}. Or more generally, asymptotic expansions of the function
\begin{equation}\label{functF}
F(z,s,a):= \Phi (z,s,a)-\frac{\mathrm{Li}_s(z)}{z^a},
\end{equation}
that generalizes the function $\eta(z,s,m-1)$ defined in \eqref{eq:lerch:sum}, from integer to complex values of the variable $m$: $\eta(z,s,m-1)=-z^mF(z,s,m)$.
Complete asymptotic expansions, including error bounds, of $\Phi(z,s,a)$ for large $a$ have been
investigated in \cite{MR2086542}. In particular, for $\Re a>0$, $\Re s>0$ and
$z\in\mathbb{C}\setminus[1,\infty)$, we have that, for arbitrary $N\in\mathbb{N}$ \cite[Theorem
1]{MR2086542},
\begin{equation}
    \Phi\left( z,s,a \right)
    =
    \sum_{n=0}^{N-1}
    c_{n}(z)
    \frac{
        \pochhammer{s}{n}
    }{
        a^{n+s}
    }
    +
    \bigO{
        a^{-N-s}
    }
    ,
    \label{eq:lerch:asy}
\end{equation}
as \(|a| \to \infty\). In this formula \(\pochhammer{s}{n} \eqd s(s+1)\dots(s+n-1)\) is the
Pochhammer symbol, $c_0(z)=(1-z)^{-1}$ and, for $n=1,2,3,...$,
\begin{equation}
    c_{n}(z)
    \eqd
    \frac{
        (-1)^{n}
       \polylog_{-n}(z)
    }{
        n !
    }.
    \label{eq:xi}
\end{equation}
%
From the identities \eqref{eq:lerch:sum} and \eqref{eq:lerch:asy} we have that, for all \(N \in \N\), $z\notin[1,\infty)$ and $\Re s>0$ where expansion \eqref{eq:lerch:asy} is valid,
\begin{equation}
    \eta(z,s,m-1) = 
    \polylog_{s}(z)
    -
    \frac{z^{m}}{m^{s}}
    \left[
        \sum_{n=0}^{N-1}
        c_{n}(z)
        \frac{
            \pochhammer{s}{n}
        }{
            m^{n}
        }
        +
        \bigO{
            m^{-N}
        }
    \right]
    ,
    \label{eq:eta:asy}
\end{equation}
as \(m \to \infty\). Unfortunately, expansion \eqref{eq:lerch:asy} has not been proved for $z\in[1,\infty)$, and then, in principle, the above expansion of $ \eta(z,s,m-1)$ does not hold in the domain of the variable $z$ where the approximation of \(\eta(z, s, a)\) has a greater interest. Had the expansion \eqref{eq:lerch:asy} been proved for \(z \ge 1\), the asymptotic computations in \eqref{eq:X:T} would have become unnecessary and an arbitrarily precise approximation could be achieved automatically from \eqref{eq:eta:asy}. 

However, to our surprise, it seems that the expansion \eqref{eq:eta:asy} is still valid when \(z > 1\). An argument that supports this claim is the following. On the one hand, assuming for the moment that \eqref{eq:eta:asy} holds for \(z = 2\), then
\begin{align}
    &   
    \eta(2, -1, m) = \sum_{n=1}^{m} n 2^{n} = (m-1) 2^{m+1} + 2 + \bigO{m^{-N}}
    ,
    &
    N \in \N
    .
    \label{eq:eta:2:asy}
\end{align}
On the other hand, using summation by parts \cite[pp.~56]{MR1397498}, it is easy to see that
\begin{equation}
    \eta(2, -1, m) = \sum_{n=1}^{m} n 2^{n} = (m-1) 2^{m+1} + 2.
    \label{eq:eta:2}
\end{equation}
Thus expansion \eqref{eq:eta:asy} seems to be correct for $z=2$. Numerical experiments further suggest
that this is also true for other values of \(z > 1\).

Then, the purpose of this paper is to show that expansion \eqref{eq:eta:asy} holds for $z> 1$. More generally, to derive an expansion of $F(z,s,a)$ for large $\Re a$ with $\Re s>0$ and $z\ge 1$. 

\end{section}

\begin{section}{An expansion of  \texorpdfstring{$\Phi(z,s,a)$}{Phi(z,s,a)} for large
        \texorpdfstring{$\Re a$}{Re(a)} and \texorpdfstring{$z\ge 1$}{z >= 1}}
        
The main result of the paper is given in Theorem 1 below. In order to formulate Theorem 1, we need to consider the function
\begin{equation}
    f(z,x,a):=\frac{1-(ze^{-x})^{1-a}}{1-ze^{-x}},
    \label{eq:f}
\end{equation}
and its Taylor coefficients $C_n(z,a)$ at $x=0$. We also need the two following Lemmas.
        
\begin{lemma} For $a,z\in\mathbb{C}$ and $n=0,1,2,...$, 
\begin{equation}\label{Cn}
C_n(z,a):=c_n(z)-z^{1-a}p_n(z,a), \hskip 2cm p_n(z,a):=\sum_{k=0}^n\frac{c_{n-k}(z)}{k!}(a-1)^k,
\end{equation}
where, for $z\ne 1$, the coefficients $c_n(z)$ have been introduced in \eqref{eq:xi}: $c_0(z)=(1-z)^{-1}$ and, for $n=1,2,3,...$,
\begin{equation}
    c_{n}(z)
    \eqd
    \frac{
        (-1)^{n}
       \polylog_{-n}(z)
    }{
        n !
    }.
    \label{eq:xiz}
\end{equation}
For $n=1,2,3,...$, the coefficients $c_n(z)$ may be computed recursively in the form $nc_n(z)=-zc'_{n-1}(z)$. Observe that $p_n(z,a)$ are polynomials of degree $n$ in the variable $a$.

\noindent
For $z=1$, the definition \eqref{Cn} must be understood in the limit sense. More precisely, \(C_{0}(1,a)=1-a\) and, for $n=1,2,3,...$,
\begin{equation}
C_{n}(1,a) 
    =
    \dfrac{
        B_{n+1}-B_{n+1}(a-1)-(n+1) (a-1)^{n}
    }{
        (n+1)!
    }
    ,
    \label{eq:c:z1}
\end{equation}
where \(B_{n}\) are the \emph{Bernoulli numbers} and $B_n(a)$ the Bernoulli polynomials \cite[\S24.2]{DLMFbernu}.
\end{lemma}

\begin{proof}
Formulas \eqref{Cn}-\eqref{eq:xiz} may be derived by combining the Taylor expansions at $x=0$ of $e^{(a-1)x}$ and $(1-ze^{-x})^{-1}$. Formula \eqref{eq:c:z1} follows from the generating function of the Bernoulli polynomials $B_n(a)$ \cite[Eq.~7.81]{MR1397498},
\begin{equation}
    \frac{e^{(a-1)x}}{1-e^{-x}}
    =
    \sum_{n = 0}^{\infty} B_{n}(a) \frac{x^{n-1}}{n!}
    ,
    \label{eq:f:z1}
\end{equation}
the definition of the Bernoulli numbers $B_n=B_n(0)$,
and the Taylor expansion of \(1-e^{-x(1-a)}\) at $x=0$. The derivation of the recursion $nc_n(z)=-zc'_{n-1}(z)$ is straightforward (see \cite{MR2086542}).
\end{proof}

Apart from the explicit form of the coefficients $C_n(z,a)$ given above, we may compute them by means of the recurrence relation given in the following lemma.

\begin{lemma} For $z\ne 1$ we have that
$\displaystyle{C_0(z,a)=\frac{1-z^{1-a}}{1-z}}$
and, for $n=1,2,3,...$,
$$
C_{n}(z,a)= \frac{z}{1-z}\left[\sum_{k=0}^{n-1}\frac{(-1)^{n-k}}{(n-k)!}C_k(z,a)-\frac{(a-1)^n}{n!z^a}\right].
$$
For $z=1$ we have
$C_0(1,a)=1-a,$
and, for $n=1,2,3,...$, 
$$
\displaystyle{
C_{n}(1,a)= -\frac{(a-1)^{n+1}}{(n+1)!}-\sum_{k=0}^{n-1}\frac{(-1)^{n-k}}{(n+1-k)!}C_k(1,a).}
$$

\end{lemma}

\begin{proof}
Replace
$$
e^{-x} = \sum_{k=0}^\infty\frac{(-x)^k}{k!} \hskip 1cm \text{and}\hskip 1cm f(z,x,a)=\sum_{k=0}^\infty C_k(z,a) x^k
$$
into the identity $[1-ze^{-x}]f(z,x,a)=1-(ze^{-x})^{1-a}$ and equate the coefficients of equal powers of $x$.
\end{proof}

\begin{theorem}
    \label{thm:lerch}
    For fixed \(N \in \N\), \(\Re a>1\) and \(\Re s >0\),
    \begin{equation}
        F(z,s,a)
        \eqd
        \Phi (z,s,a)-\frac{\mathrm{Li}_s(z)}{z^a}
        =
        \sum_{n=0}^{N-1}C_n(z,a)\frac{(s)_n}{a^{n+s}}+R_N(z,s,a),
        \label{eq:F:asy}
    \end{equation}
with $C_n(z,a)$ given in the previous lemmas and, for \(z> 1\),
  \begin{align}
        &
        R_{N}(z,s,a)
        =
        \bigO{(\Re a)^{1-N-s}+a z^{-\Re a}}
        ,
        &
        \Re a \to \infty
        .
        \label{eq:R}
    \end{align}
This means that expansion \eqref{eq:F:asy} has an asymptotic character for large $\Re a$ when $z>1$. Moreover, expansion \eqref{eq:F:asy} is convergent for \(a=m=2,3,4,...\) and \(z\ge 1\); i.e., $R_{N}(z,s,m)\to 0$ as $N\to\infty$ and
    \begin{equation}
        F(z,s,m)
        :=-\frac{1}{z^m}\sum_{k=1}^{m-1}\frac{z^k}{k^s}=
       \sum_{n=0}^{\infty}C_n(z,m)\frac{(s)_n}{m^{n+s}}, \hskip 2cm z\ge 1.
        \label{eq:F:conv}
    \end{equation}

    \end{theorem}

\begin{proof}
Using the integral representation (2) of $\Phi (z,s,a)$ given in \cite{MR2086542} and the integral
representation \cite[\S25.12.11]{DLMFlerch} of the polylogartithm,
$$
\mathrm{Li}_s(z)=\frac{z}{\Gamma(s)}\int_0^\infty\frac{x^{s-1}}{e^x-z}\,\mathrm dx,
$$
we find the following integral representation of the function $F(z,s,a)$ defined in \eqref{functF}:
\begin{equation}
    F(z,s,a):=\Phi (z,s,a)-\frac{\mathrm{Li}_s(z)}{z^a}=\frac{1}{\Gamma(s)}\int_0^\infty x^{s-1}e^{-ax}f(z,x,a)\,\mathrm dx,
    \label{eq:F}
\end{equation}
valid for $\Re a>0$, $\Re s>0$ and $z\in\mathbb{C}$, with $f(z,x,a)$ given in \eqref{eq:f}.
In principle, the left hand side of \eqref{eq:F} is an analytic function of $z$ in the disk $\vert z\vert<1$. Then, the right hand side of this equation defines the analytic continuation of $F(z,s,a)$ in the variable $z$ to the cut complex plane $\mathbb{C}\setminus(-\infty,0]$. 

The function $f(z,x,a)$ has the following Taylor expansion at $x=0$:
\begin{equation}\label{taylor}
f(z,x,a)=\sum_{k=0}^{n-1}C_k(z,a)x^k+r_n(z,x,a),
\end{equation}
where the coefficients $C_n(z,a)$ are given in \eqref{Cn}-\eqref{eq:xiz}-\eqref{eq:c:z1}. Replacing the function $f(z,x,a)$ in the integral \eqref{eq:F} by its Taylor expansion \eqref{taylor}, and interchanging sum and integral, we obtain \eqref{eq:F:asy} with
\begin{equation}\label{boundR}
R_n(z,s,a):=\frac{1}{\Gamma(s)}\int_0^{\infty} x^{s-1}e^{-ax}r_n(z,x,a)\,\mathrm dx.
\end{equation}
Using the Cauchy integral formula for the remainder $r_n(z,x,a)$ we find
\begin{equation}\label{remain}
r_n(z,x,a)=\frac{x^n}{2\pi i}\int_C\frac{f(z,w,a)}{(w-x)w^n}\,\mathrm{d}w, \hskip 2cm
f(z,w,a):=\frac{1-(ze^{-w})^{1-a}}{1-ze^{-w}},
\end{equation}
where we choose $C$ to be a closed loop that encircles the points $w=0$ and $w=x$, it is
traversed in the positive direction, and is inside the region $U:=\lbrace w\in\mathbb{C}, \Re w>-W$,
$\vert \Im w\vert<W\rbrace$, for some arbitrary but fixed $W \in (0,2\pi)$. ($U$ is an infinity rectangular region around the
positive real line $[0,\infty)$ of width $2W$ (see \cite{MR2086542} for further details). 
\autoref{fig:C} gives an example of \(U\) and \(C\).

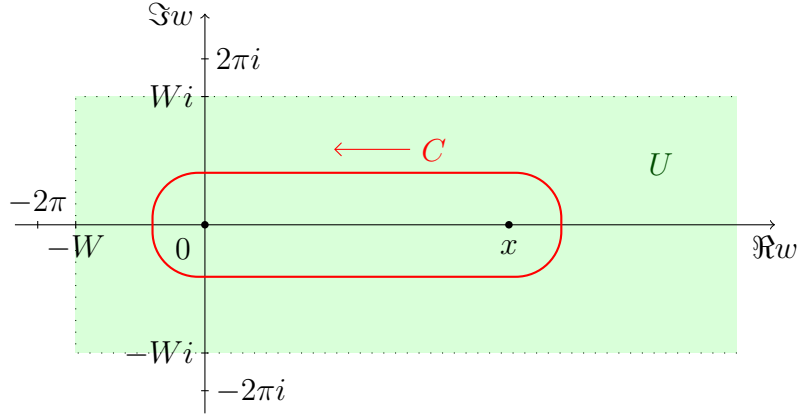
\begin{figure}[ht]
    \begin{center}
        \begin{tikzpicture}[scale=1] 

            \pgfmathsetmacro{\W}{1.7}
            \pgfmathsetmacro{\Wend}{7}
            \pgfmathsetmacro{\twopi}{2.2}

            \fill[green!15!white] (-\W,-\W) -- (-\W,\W) -- (\Wend,\W) -- (\Wend,-\W);
            \draw[loosely dotted,green!10!black] (\Wend,-\W) -- (-\W,-\W) -- (-\W,\W) -- (\Wend,\W);

            \draw[->] (-2.5,0) -- (7.5,0) node[anchor=north] {$\Re w$};
            \draw[->] (0,-2.5) -- (0,2.8) node[anchor=east] {$\Im w$};

            \node[text=black!70!green] at (6,0.8) {$U$};

            \node[anchor=east] at (0,\W) {$W i$};
            \node[anchor=east] at (0,-\W) {$-W i$};
            \draw (-0.05,\W)--(0.05,\W);
            \draw (-0.05,-\W)--(0.05,-\W);
            \node[anchor=north] at (-\W,-0) {$-W$};
            \draw (-\W,-0.05)--(-\W,0.05);

            \node[anchor=west] at (0,\twopi) {$2 \pi i$};
            \node[anchor=west] at (0,-\twopi) {$-2 \pi i$};
            \draw (-0.05,\twopi)--(0.05,\twopi);
            \draw (-0.05,-\twopi)--(0.05,-\twopi);
            \node[anchor=south] at (-\twopi,-0) {$-2 \pi$};
            \draw (-\twopi,-0.05)--(-\twopi,0.05);

            \node[circle,fill,inner sep=1pt,label=below:$x$] at (4,0) (nx) {};
            \node[circle,fill,inner sep=1pt,label=below left:$0$] at (0,0) (n0) {};

            \node[draw,thick, fit=(nx) (n0),inner sep=18pt, label={[red]above right:$C$},rounded
            corners=0.6cm,red] {};

            \draw[arrows={-angle 90},red] (2.7,1)--(1.7,1);

        \end{tikzpicture}
    \end{center}
    \caption{The integration loop \(C\) and the region \(U\)}\label{fig:C}
\end{figure}

The function $f(z,w,a)$ is continuous in the variable $w$ for $w\in U$. The singularities of this
function are $w=\log z+2i\pi n$, $n\in \Z\setminus\lbrace 0\rbrace$ and are located outside $U$.
Define $b:=\lfloor\Re a\rfloor$ and $\beta:=a-b$ and write
\begin{equation}
\begin{aligned}\label{decom}
f(z,w,a)&= 
\frac{1-(ze^{-w})^{1-a}}{1-ze^{-w}}=\frac{1-(ze^{-w})^{-\beta}+(ze^{-w})^{-\beta}-(ze^{-w})^{1-b-\beta}}{1-ze^{-w}}
\\ &= \frac{1-(ze^{-w})^{-\beta}}{1-ze^{-w}}+(ze^{-w})^{-\beta}\frac{1-(ze^{-w})^{1-b}}{1-ze^{-w}}.
\end{aligned}
\end{equation}
We have that
\begin{equation}
   \frac{1-(ze^{-w})^{1-b}}{1-ze^{-w}}
    =\sum_{k=0}^{b-2}\frac{(ze^{-w})^{-k}-(z e^{-w})^{-k-1}}{1-ze^{-w}}
    = -\sum_{k=1}^{b-1}\left(\frac{e^w}{z}\right)^k.
    \label{eq:f:b}
\end{equation}
The following bounds are valid for $w\in U$ and a certain constant $M_0>0$ independent of $b$ and $\Re w$:
\begin{equation}\label{bound}
    \left\{
        \begin{array}{lll}
            \displaystyle{
                \left\vert \frac{1-(ze^{-w})^{-\beta}}{1-ze^{-w}}\right\vert,
                \hskip 5mm \left\vert(ze^{-w})^{-\beta}\right\vert 
            } 
            & \le M_0\, e^{\Re \beta\Re w}, 
            & \text{for} \hskip 2mm  0\le\Re\beta<1, \\
            \displaystyle{\left\vert\sum_{k=1}^{b-1}\left(\frac{e^w}{z}\right)^k\right\vert} 
            &
            \le \displaystyle{(b-1)\frac{e^{\Re w}}{z}}, 
            &\text{for} \hskip 2mm \Re w\le\log z, \\
            \displaystyle{
                \left\vert\sum_{k=1}^{b-1}\left(\frac{e^w}{z}\right)^k\right\vert} 
            &\le \displaystyle{(b-1)\left(\frac{e^{\Re w}}{z}\right)^{b-1}},
            &\text{for}\hskip 2mm \Re w\ge\log z.
        \end{array}
    \right.
\end{equation}
%
Therefore, from \eqref{decom}, \eqref{eq:f:b} and \eqref{bound} we have that, for any $w\in U$ and $\Re a>1$,
$$
\vert f(z,w,a)\vert\le M\vert a\vert[e^{\Re w}+ z^{1-\Re a}e^{(\Re a-1)\Re w}],
$$
for a certain constant $M>0$ independent of $\Re a$ and $\Re w$.
The path $C$ in \eqref{remain} may be chosen in such a way that $\Re w\le x+1/\Re a$. Then, from \eqref{remain}, we find the following bound for the remainder $r_n(z,x,a)$:
$$
\vert r_n(z,x,a)\vert\le M_n \vert a\vert x^n[e^{x}+ z^{1-\Re a}e^{x(\Re a-1)}], \hskip 1cm \text{for}\hskip 5mm x\ge 0,
$$
where $M_n$ is a certain positive constant that depends on the geometry of the path $C$ chosen in \eqref{remain} and $\Im a$ but not on $\Re a$. Then, from \eqref{boundR},
\begin{equation}
    \vert R_n(z,s,a)\vert\le\frac{M_n\vert a\vert}{\vert\Gamma(s)\vert}\int_0^{\infty} x^{n+s-1}[e^{(1-\Re a)x}+ z^{1-\Re a}e^{-x}]\,\mathrm dx=\mathcal{O}((\Re a)^{1-n-s}+a\, z^{1-\Re a}).
    \label{eq:R:1}
\end{equation}
This proves \eqref{eq:R} and the asymptotic character of the expansion \eqref{eq:F:asy} for large $\Re a$ when $z>1$.

Finally, we prove formula \eqref{eq:F:conv}. When $a=m\in\lbrace 2,3,4,...\rbrace$, the Taylor coefficients $C_n(z,m)$ of $f(z,x,m)$ at $x=0$ are given in \eqref{Cn}-\eqref{eq:xiz}-\eqref{eq:c:z1} with $a=m$. But we may derive a simpler formula for $C_n(z,m)$. We have that
\begin{equation}
    f(z,x,m) 
    :=\frac{1-(ze^{-x})^{1-m}}{1-ze^{-x}}
    =\sum_{k=0}^{m-2}\frac{(ze^{-x})^{-k}-(z e^{-x})^{-k-1}}{1-ze^{-x}}
    = -\sum_{k=1}^{m-1}\frac{e^{kx}}{z^k}.
    \label{eq:f:m}
\end{equation}
Then, the Taylor coefficients $C_n(z,m)$ of $f(z,x,m)$ at $x=0$ are the sum of the Taylor coefficients of $z^{-k}e^{kx}$, that is,
$$
C_n(z,m)=-\frac{1}{n!}\sum_{k=1}^{m-1}\frac{k^n}{z^k}=\frac{1}{n!}\left[z^{-m}\Phi
    (z^{-1},-n,m)-\mathrm{Li}_{-n}(z^{-1})\right].
$$
Since
$$
\vert C_n(z,m)\vert=\frac{1}{n!}\sum_{k=1}^{m-1}\frac{k^n}{z^k}\le
\frac{(m-1)^n}{n!}\sum_{k=1}^{m-1}\frac{1}{z^k}\le\frac{(m-1)^{n+1}}{n!},
$$
we have that
$$
\sum_{n=0}^\infty\vert C_n(z,m)\vert\int_0^\infty x^{n+\Re s-1}e^{-mx} dx\le
\frac{m-1}{m^{\Re s}}\sum_{n=0}^\infty\frac{\Gamma(n+\Re s)}{n!}\left(\frac{m-1}{m}\right)^n<\infty.
$$
Therefore, when $a=m\ge 2$, we may replace $f(z,x,m)$ into the integral \eqref{eq:F} by its Taylor expansion \eqref{taylor} and interchange sum and integral. This proves \eqref{eq:F:conv}.
\end{proof}

\begin{corollary}
    \label{cor:eta}
    For fixed \(N \in \N\), \(z>1\), $m\in\mathbb{N}$ and \(\Re s>0\),
    \begin{equation}
        \label{eq:eta:asy:1}
        \eta(z,s,m-1) = 
        -
        \frac{z^{m}}{m^{s}}
        \left[
            \sum_{n=0}^{N-1}
            \left[c_n(z)+z^{1-m}p_n(z,m)\right]
            \frac{
                \pochhammer{s}{n}
            }{
                m^{n}
            }
            +
            \bigO{m^{1+s-N}+m^{s+1} z^{-m}}
        \right]
        ,
    \end{equation}
    as \(m \to \infty\).   Moreover, for $z\ge 1$,
    \begin{equation}
        \label{eq:eta:cov}
        \eta(z,s,m-1) = 
        -
        \frac{z^{m}}{m^{s}}
            \sum_{n=0}^{\infty}
            \left[c_n(z)+z^{1-m}p_n(z,m)\right]
            \frac{
                \pochhammer{s}{n}
            }{
                m^{n}
            }
        .
    \end{equation}
\end{corollary}

\end{section}

\begin{section}{Final remarks and numeric experiments}

\begin{myRemark}
The integral representation \eqref{inte} of $\Phi(z,s,a)$ is not valid for $z\in[1,\infty)$ because of
the pole of the integrand at $x=\log z$. This pole is removed by the subtraction of the function
$x^{s-1}(e^x-z)^{-1}$ to the integrand. We obtain in this way the integral representation
\eqref{eq:F} of the function $F(z,s,a):=\Phi(z,s,a)-z^{-a}\mathrm{Li}_s(z)$, free of the pole $x=\log z$ and valid for $z\in\mathbb{C}\setminus(-\infty,0]$.
\end{myRemark}

\begin{myRemark}
    Since \(f(z,x,0)\equiv 1\), we have \(C_{0}(z,0)=1\) and \(C_{n}(z,0) \equiv 0\) for $n=1,2,3,...$. Thus by \eqref{Cn}, 
    for all \(n \in \N\) and \(z \ne 1\),
    \begin{equation}
        c_{n}(z)=\frac{z}{1-z} \sum_{k=0}^{n-1} \frac{(-1)^{n-k}c_{k}(z)}{(n-k)!},
        \label{eq:recur:eta:4}
    \end{equation}
    which is equivalent to
    \begin{equation}
        \polylog_{-n}(z) = 
        \frac{z}{(1-z)^2}
        +
        \frac{z}{1-z} \sum_{k=1}^{n-1} \binom{n}{k} \polylog_{-k}(z).
        \label{eq:recur:Li}
    \end{equation}
\end{myRemark}

\begin{myRemark}
Observe that the terms of the expansion \eqref{eq:F:asy} are not a pure Poincar\'e expansion in the asymptotic sequence $a^{-k-s}$, as the coefficients $C_k(z,a)=c_k(z)+z^{1-a}p_k(z,a)$, $z\ge 1$, depend on $a$ ($p_k(z,a)$ is a polynomial of degree $k$ in $a$). For $z>1$ these coefficients are separable and then we may write the expansion \eqref{eq:F:asy} in the form
\begin{equation}
    \begin{aligned}
        F(z,s,a)
        &
        =\sum_{k=0}^{n-1}c_k(z)\frac{(s)_k}{a^{k+s}}+z^{1-a}\sum_{k=0}^{n-1}p_k(z,a)\frac{(s)_k}{a^{k+s}}+R_n(z,s,a)
        \\
        &
        =\sum_{k=0}^{n-1}c_k(z)\frac{(s)_k}{a^{k+s}}+\mathcal{O}((\Re a)^{1-n-s}+az^{1-\Re a}).
    \end{aligned}
    \label{eq:F:asy:2}
\end{equation}
%
The first expansion is the expansion \eqref{eq:F:asy} of $\Phi(z,a,s)$, valid for $z\notin[1,\infty)$. The second one is an
exponentially small correction; when $z$ is very large, it is a small correction, but when $z$ is close to 1, it is not negligible.
\end{myRemark}

\begin{myRemark}
Using the summation by parts formula \cite[\S2.10.9]{DLMFaa}, we have that
\begin{equation}
    \begin{aligned}
        \eta(z,s,m)
        &
        =
        \frac{z^{m+1}}{z-1}
        \frac{1}{(m+1)^{s}}
        -
        \frac{z}{z-1}
        +
        \frac{z}{z-1}
        \sum_{k=1}^{m}
        \left(  
            1-\left( 1+k^{-1} \right)^{-s}
        \right)
        \frac{z^{k}}{k^{s}}
        \\
        &
        =
        \frac{z^{m+1}}{z-1}
        \frac{1}{(m+1)^{s}}
        -
        \frac{z}{z-1}
        +
        \frac{z}{z-1}
        \sum_{k=1}^{m}
        \left(  
            \sum_{n=1}^{\infty}
            \frac{\pochhammer{s}{n}(-1)^{n-1}}{n!}
            \frac{1}{k^{n}}
        \right)
        \frac{z^{k}}{k^{s}}
        \\
        &
        =
        \frac{z^{m+1}}{z-1}
        \frac{1}{(m+1)^{s}}
        -
        \frac{z}{z-1}
        +
        \frac{z}{z-1}
        \sum_{n=1}^{\infty}
        \frac{\pochhammer{s}{n}(-1)^{n-1}}{n!}
        \eta(z,s+n,m)
        .
    \end{aligned}
    \label{eq:by:parts}
\end{equation}
Thus, expansion \eqref{eq:eta:asy} could also be proved by induction on \(N\) using the above identity.
\end{myRemark}

\begin{myRemark} Note that, for \(z=1\),
    \begin{equation}
        -F(1,s,m)
        =
        \eta\left(1, s, m-1 \right)
        =
        \sum_{n=1}^{m-1} \frac{1}{n^{s}}
        =
        \sum_{n=1}^{\infty} \frac{1}{n^{s}}
        -
        \sum_{n=m}^{\infty} \frac{1}{n^{s}}
        =
        \zeta(s)-\zeta(s,m)
        ,
        \label{eq:harm}
    \end{equation}
    where \(\zeta(s)\) denotes the \emph{Riemann zeta function} \cite[\S25.2]{DLMFlerch} and
    \(\zeta(s,m)\) denotes the \emph{Hurwitz zeta function} \cite[\S25.11]{DLMFlerch}.
    Thus \eqref{eq:F:conv} gives a new series representation of \(\zeta(s,m)\) for \(m \in \N\), as
    \begin{equation}
        \begin{aligned}
            \zeta(s,m) 
            &
            = 
            \zeta(s)+\sum_{k=0}^\infty C_{k}(1,m) \frac{\pochhammer{s}{k}}{m^{k+s}}
            \\
            &
            =
            \zeta(s)+\frac{1}{m^{s}} 
            \left( 
                1-m
                +
                \sum_{k=1}^\infty 
                \dfrac{
                    B_{k+1}-B_{k+1}(m-1)-(k+1) (m-1)^{k}
                }{
                    (k+1)!
                }
                \frac{\pochhammer{s}{k}}{m^{k}}
            \right)
            \\
            &
            =
            \zeta(s)+\frac{1}{m^{s}} 
            \left( 
                2-m-m^{s}
                +
                \sum_{k=1}^\infty 
                \dfrac{
                    B_{2k}
                }{
                    (2k)!
                }
                \frac{\pochhammer{s}{2k-1}}{m^{2k-1}}
                -
                \sum_{k =1}^\infty 
                \dfrac{
                    B_{k+1}(m-1)
                }{
                    (k+1)!
                }
                \frac{\pochhammer{s}{k}}{m^{k}}
            \right)
            .
        \end{aligned}
        \label{eq:hurwitz}
    \end{equation}
    On the other hand, using the \emph{Euler-Maclaurin's summation formula} \cite[\S9.5]{MR1397498}, or the asymptotic expansion of
    \(\zeta(s,m)\) in \cite[Eq.~25.11.43]{DLMFlerch}, we have, for \(s > 1\),
    \begin{equation}
        F(1,s,m)
        =
        -\zeta(s)
        +
        \frac{m^{1-s}}{s-1}
        +\frac{m^{-s}}{2}
        +
        \sum_{k=1}^{n}
        \frac{B_{2 k}}{(2k)!}
       \frac{ \pochhammer{s}{2 k-1}}{m^{2 k+s-1}}
        +
        \hat R_{n}(s,m)
        ,
        \label{eq:F:z1:asy}
    \end{equation}
    with
    \[
        \left|\hat R_{n}(s,m)\right| \le
            \frac{
                \vert B_{2 n + 2}\vert
            }{
                (2 n + 2) !
            }
            \frac{\vert\pochhammer{s}{2n+1}\vert}{m^{2n+s+1}} 
        .
    \]
    %
\end{myRemark}

\bigskip
\noindent
The following tables and pictures show some numerical experiments about the accuracy of the
approximations given in Theorem 1. In the tables we compute the absolute value of the relative error $\bar R_n(z,s,a)$ in the approximation \eqref{eq:F:asy}, defined in the form
$$
\bar R_n(z,s,a):=1-\frac{\sum_{k=0}^{n-1}C_k(z,a)\dfrac{(s)_k}{a^{k+s}}}{F(z,s,a)}.
$$
In \autoref{table1} and \autoref{figure2} we evaluate the Lerch's transcendent, the Polylogarithm and all the approximations with the symbolic manipulator {\it Wolfram Mathematica 10.4}.
%

%
%
%
%

\begin{table}[ht]
    \setlength{\extrarowheight}{5pt}
    \small
    \centering
    \begin{minipage}{0.48\textwidth}
        \centering
        \begin{tabular}{|c|c|c|c|}
            \multicolumn{4}{c}{$z= 2,\, s=1$}
            \\ \hline $n$ & $a=5$ & $a=10$ & $a=20$
            \\ \hline
            $5$ &7.87e-2& 2.22e-2 & 6.21e-4
            \\ \hline
            $10$ &2.13e-2& 7.55e-3 & 7.36e-5
            \\ \hline
            $15$ &6.69e-3& 3.68e-3 & 3.24e-5
            \\ \hline
        \end{tabular}
    \end{minipage}
    \hfill
    \begin{minipage}{0.48\textwidth}
        \centering
        \begin{tabular}{|c|c|c|c|}
            \multicolumn{4}{c}{$z=5,\, s=2$}
            \\ \hline $n$ & $a=5$ & $a=10$ & $a=20$
            \\ \hline
            $5$ &8.36e-2 & 2.57e-3& 5.87e-5
            \\ \hline
            $10$ &2.82e-2& 2.89e-4 & 1.21e-7
            \\ \hline
            $15$ &1.13e-2& 1.23e-4 & 2.66e-9
            \\ \hline
        \end{tabular}
    \end{minipage}
    \vskip 8pt
    \begin{minipage}{0.48\textwidth}
        \centering
        \begin{tabular}{|c|c|c|c|}
            \multicolumn{4}{c}{$z= 2,\, s=2$}
            \\ \hline $n$ & $a=10+i$ & $a=30+i$ & $a=50+i$
            \\ \hline
            $5$ &1.60e-1& 3.67e-4 & 2.41e-5
            \\ \hline
            $10$ &9.14e-2& 1.11e-5 & 3.32e-8
            \\ \hline
            $15$ &5.92e-2& 3.62e-6 & 4.75e-10
            \\ \hline
        \end{tabular}
    \end{minipage}
    \hfill
    \begin{minipage}{0.48\textwidth}
        \centering
        \begin{tabular}{|c|c|c|c|}
            \multicolumn{4}{c}{$z=5,\, s=3$}
            \\ \hline $n$ & $a=10+i$ & $a=30+i$ & $a=50+i$
            \\ \hline
            $5$ &9.59e-3 & 2.52e-5& 1.91e-6
            \\ \hline
            $10$ &2.37e-3& 1.04e-8 & 5.78e-11
            \\ \hline
            $15$ &1.43e-3& 3.47e-11 & 1.35e-14
            \\ \hline
        \end{tabular}
    \end{minipage}
    \caption{
        The relative error in the approximation \eqref{eq:F:asy} for Lerch's transcendent
    }
    \label{table1}
\end{table}

\begin{figure}[ht]
    \small
    \centering
    \begin{minipage}{0.48\textwidth}
        \centering
        \includegraphics[width=\textwidth]{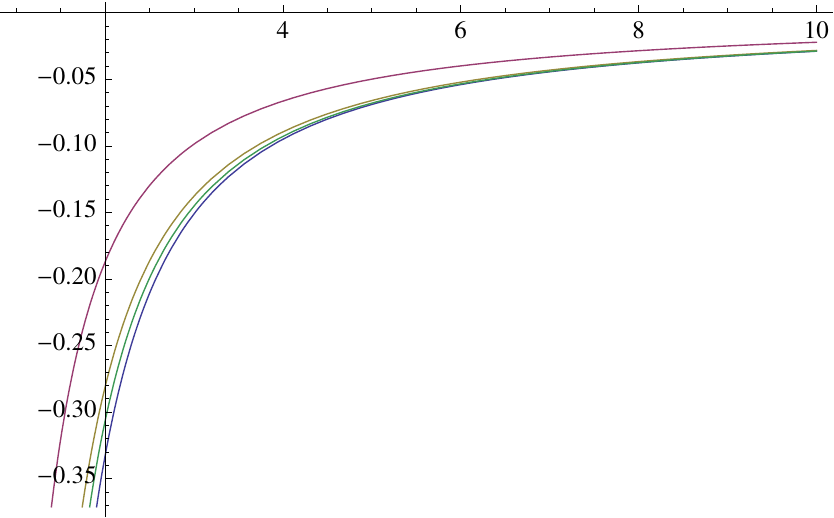}
        (i) $a=5$, $s=1$, $z\in[1,10]$, $n=1,3,5$.
    \end{minipage}
    \hfill
    \begin{minipage}{0.48\textwidth}
        \centering
        \includegraphics[width=\textwidth]{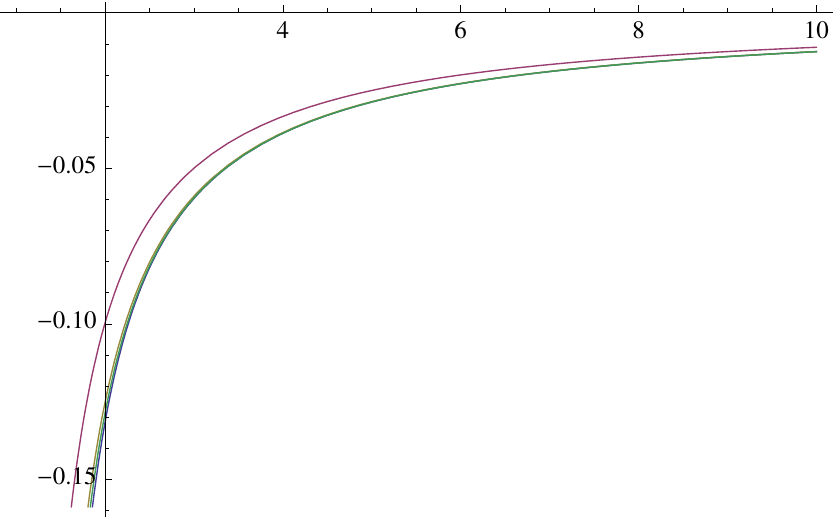}
        (ii) $a=10$, $s=1$, $z\in[1,10]$, $n=1,3,5$.
    \end{minipage}
    \vskip 10pt
    \begin{minipage}{0.48\textwidth}
        \centering
        \includegraphics[width=\textwidth]{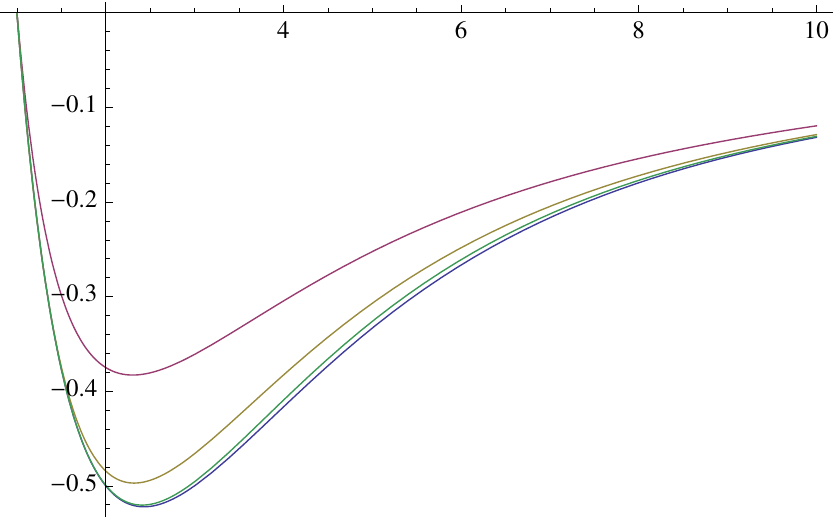}
        (iii) $z=2$, $s=1$, $a\in[1,10]$, $n=2,5,10$.
    \end{minipage}
    \hfill
    \begin{minipage}{0.48\textwidth}
        \centering
        \includegraphics[width=\textwidth]{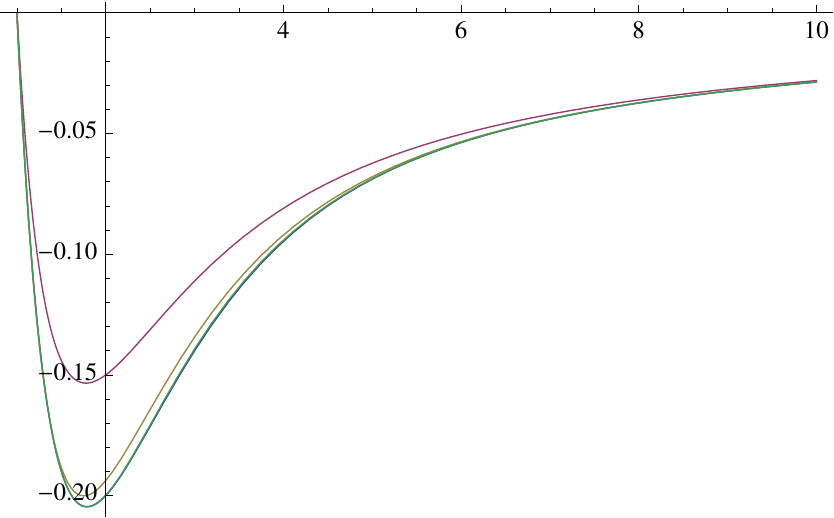}
        (iv) $z=5$, $s=1$, $a\in[1,10]$, $n=2,5,10$.
    \end{minipage}
  \caption{
  The blue line is the graphic of the function $F(z,s,a)$, whereas the red, gold and green functions represent the right 
  hand side of \eqref{eq:F:asy} for increasing values of the approximation order $n$.
  }
  \label{figure2}
\end{figure}

\end{section}


\begin{thebibliography}{15}
\providecommand{\natexlab}[1]{#1}
\providecommand{\url}[1]{\texttt{#1}}
\expandafter\ifx\csname urlstyle\endcsname\relax
  \providecommand{\doi}[1]{doi: #1}\else
  \providecommand{\doi}{doi: \begingroup \urlstyle{rm}\Url}\fi
  

\bibitem[{\relax Apostol}()]{apostol}
{\relax T. M. Apostol}.
\newblock On the Lerch zeta function.
\newblock \emph{Pacific J. Math.}, 1\penalty0 (1):\penalty0 161--167,
  1951.
  
  \bibitem[{\relax DLMFlerch}()]{DLMFlerch}
{\relax T. M. Apostol}.
\newblock {{Zeta and Related Functions, in: NIST Handbook of Mathematical Functions}}.
\newblock  Cambridge University Press, Cambridge, 2010, pp. 601--616 (Chapter 25).
\newblock URL \url{http://dlmf.nist.gov/}.

\bibitem[{\relax DLMF}()]{bateman}
{\relax H. Bateman}.
\newblock {{Higher transcendental functions. Volume I}}.
\newblock McGraw-Hill, New York, 1953.

\bibitem[Cai et~al.(2018)Cai, {Devroye}, {Holmgren}, and {Skerman}]{Cai010}
X.~S. Cai, L.~{Devroye}, C.~{Holmgren}, and F.~{Skerman}.
\newblock {Cutting resilient networks}.
\newblock \emph{ArXiv e-prints}, Apr. 2018.


\bibitem[{\relax DLMFbernu}()]{DLMFbernu}
{\relax K. Dilcher}.
\newblock {{Bernoulli and Euler Polynomials, in: NIST Handbook of Mathematical Functions}}.
\newblock  Cambridge University Press, Cambridge, 2010, pp. 587--599 (Chapter 24).
\newblock URL \url{http://dlmf.nist.gov/}.

\bibitem[Erd\'elyi et~al.(1953)Erd\'elyi, Magnus, Oberhettinger, and
  Tricomi]{MR0058756}
A.~Erd\'elyi, W.~Magnus, F.~Oberhettinger, and F.~G. Tricomi.
\newblock \emph{Higher transcendental functions. {V}ols. {I}, {II}}.
\newblock McGraw-Hill Book Company, Inc., New York-Toronto-London, 1953.
\newblock Based, in part, on notes left by Harry Bateman.

\bibitem[Ferreira and L\'opez(2004)]{MR2086542}
C.~Ferreira and J.~L. L\'opez.
\newblock Asymptotic expansions of the {H}urwitz-{L}erch zeta function.
\newblock \emph{J. Math. Anal. Appl.}, 298\penalty0 (1):\penalty0 210--224,
  2004.
\newblock ISSN 0022-247X.
\newblock \doi{10.1016/j.jmaa.2004.05.040}.

\bibitem[Graham et~al.(1994)Graham, Knuth, and Patashnik]{MR1397498}
R.~L. Graham, D.~E. Knuth, and O.~Patashnik.
\newblock \emph{Concrete mathematics}.
\newblock Addison-Wesley Publishing Company, Reading, MA, second edition, 1994.
\newblock ISBN 0-201-55802-5.
\newblock A foundation for computer science.

\bibitem[Janson(2004)]{janson04}
S.~Janson.
\newblock Random records and cuttings in complete binary trees.
\newblock In \emph{Mathematics and computer science. {III}}, Trends Math.,
  pages 241--253. Birkh\"auser, Basel, 2004.

\bibitem[{\relax Katsurada}()]{katsuradaii}
M. Katsurada.
\newblock On an asymptotic formula of Ramanujan for a certain theta-type series.
\newblock \emph{Acta Arith.}, XCVII.2 \penalty0 (1):\penalty0 157--172,
2001.

\bibitem[{\relax Kluss}()]{kluss}
D. Klusch.
\newblock Asymptotic equalities for the Lipschitz-Lerch zeta-function.
\newblock \emph{Arch. Math.}, 49 \penalty0 (1):\penalty0 38--43,
1987.

\bibitem[{\relax Lerch}()]{lerch}
M. Lerch.
\newblock Note sur la fonction ${\cal
R}(w,x,s)=\sum_{k=0}^\infty{e^{2k\pi i x }/ (w+k)^s }$.
\newblock \emph{Acta Math.}, 11\penalty0 (1):\penalty0 19--24,
1887.

\bibitem[{\relax Lipschitz}()]{lipschitz}
R. Lipschitz.
\newblock Untersuchung einer aus vier elementen gebildeten reihe.
\newblock \emph{J. reine angew Math.}, 54\penalty0 (1):\penalty0 127--156,
1889.
  
\bibitem[{\relax DLMFaa}()]{DLMFaa}
{\relax F. W. J. Olver and R. Wong}.
\newblock {{Asymptotic Approximations, in: NIST Handbook of Mathematical Functions}}.
\newblock  Cambridge University Press, Cambridge, 2010, pp. 41--70 (Chapter 2).
\newblock URL \url{http://dlmf.nist.gov/}.
  
\bibitem[{\relax Srivastava}()]{sriva}
H. M. Srivastava and J. Choi
\newblock \emph{Series Associated with the Zeta and Related Functions}.
\newblock \emph{Kluwer Acad. Pub., London}, 2001.
  


\end{thebibliography}

\end{document}